\documentclass{article}
\usepackage[utf8]{inputenc}
\usepackage{amsmath,amsthm,amssymb}
\usepackage{tikz,xcolor,verbatim,hyperref}
\usetikzlibrary{shapes.geometric}
\newtheorem{theorem}{Theorem}
\newtheorem{observation}[theorem]{Observation}
\newtheorem{corollary}[theorem]{Corollary}
\newtheorem{lemma}[theorem]{Lemma}
\newtheorem{conjecture}[theorem]{Conjecture}

\newenvironment{lrcases}
{\left\lbrace\begin{aligned}}
{\end{aligned}\right\rbrace}
  
\title{Precoloring extension in planar near-Eulerian-triangulations\thanks{Supported by project 22-17398S (Flows and cycles in graphs on surfaces) of Czech Science Foundation.  An extended abstract appeared in Proceedings of the 12th European Conference on Combinatorics, Graph Theory and Applications (EUROCOMB’23).}}

\author{Zdeněk Dvořák \and Benjamin Moore \and Michaela Seifrtová\and Robert Šámal\thanks{Computer Science Institute, Charles University, Prague, Czech Republic. The second author completed this work while at Charles University, and is now at the Institute of Science and Technology, Klosterneuburg, Austria. E-mails: \url{{rakdver,mikina,samal}@iuuk.mff.cuni.cz}, \url{Benjamin.Moore}@ist.ac.at}}

\usetikzlibrary{shapes.geometric}

\tikzset{
blackvertexv2/.style={circle, draw=black!100,fill=black!100,thick, inner sep=0pt, minimum size= 2.3mm},
blackvertex/.style={rectangle, draw=black!100,fill=black!100,thick, inner sep=0pt, minimum size= 2.3mm},
blackvertexv3/.style={regular polygon,regular polygon sides=3, draw=black!100,fill=black!100,thick, inner sep=0pt, minimum size= 2.8mm},
greenvertexv2/.style={circle, draw=black!100,fill=green!100,thick, inner sep=0pt, minimum size= 2.3mm},
greenvertex/.style={rectangle, draw=black!100,fill=green!100,thick, inner sep=0pt, minimum size= 2.3mm},
greenvertexv3/.style={regular polygon,regular polygon sides=3, draw=black!100,fill=green!100,thick, inner sep=0pt, minimum size= 2.8mm},
bluevertexv2/.style={circle, draw=black!100,fill=blue!100,thick, inner sep=0pt, minimum size= 2.3mm},
bluevertex/.style={rectangle, draw=black!100,fill=blue!100,thick, inner sep=0pt, minimum size= 2.3mm},
bluevertexv3/.style={regular polygon,regular polygon sides=3, draw=black!100,fill=blue!100,thick, inner sep=0pt, minimum size= 2.8mm},
yellowvertexv2/.style={circle, draw=black!100,fill=yellow!100,thick, inner sep=0pt, minimum size= 2.3mm},
yellowvertex/.style={rectangle, draw=black!100,fill=yellow!100,thick, inner sep=0pt, minimum size= 2.3mm},
yellowvertexv3/.style={regular polygon,regular polygon sides=3, draw=black!100,fill=yellow!100,thick, inner sep=0pt, minimum size= 2.8mm},
redvertexv2/.style={circle, draw=black!100,fill=red!100,thick, inner sep=0pt, minimum size= 2.3mm},
redvertex/.style={rectangle, draw=black!100,fill=red!100,thick, inner sep=0pt, minimum size= 2.3mm},
redvertexv3/.style={regular polygon,regular polygon sides=3, draw=black!100,fill=red!100,thick, inner sep=0pt, minimum size= 2.8mm},
tealvertexv2/.style={circle, draw=black!100,fill=teal!100,thick, inner sep=0pt, minimum size= 2.3mm},
tealvertex/.style={rectangle, draw=black!100,fill=teal!100,thick, inner sep=0pt, minimum size= 2.3mm},
tealvertexv3/.style={regular polygon,regular polygon sides=3, draw=black!100,fill=teal!100,thick, inner sep=0pt, minimum size= 2.8mm},
violetvertexv2/.style={circle, draw=black!100,fill=violet!100,thick, inner sep=0pt, minimum size= 2.3mm},
violetvertex/.style={rectangle, draw=black!100,fill=violet!100,thick, inner sep=0pt, minimum size= 2.3mm},
violetvertexv3/.style={regular polygon,regular polygon sides=3, draw=black!100,fill=violet!100,thick, inner sep=0pt, minimum size= 2.8mm},
limevertexv2/.style={circle, draw=black!100,fill=olive!100,thick, inner sep=0pt, minimum size= 2.3mm},
limevertex/.style={rectangle, draw=black!100,fill=olive!100,thick, inner sep=0pt, minimum size= 2.3mm},
limevertexv3/.style={regular polygon,regular polygon sides=3, draw=black!100,fill=olive!100,thick, inner sep=0pt, minimum size= 2.8mm},
magentavertexv2/.style={circle, draw=black!100,fill=magenta!100,thick, inner sep=0pt, minimum size= 2.3mm},
magentavertex/.style={rectangle, draw=black!100,fill=magenta!100,thick, inner sep=0pt, minimum size= 2.3mm},
magentavertexv3/.style={regular polygon,regular polygon sides=3, draw=black!100,fill=magenta!100,thick, inner sep=0pt, minimum size= 2.8mm},
dummywhite/.style={circle, draw=white!100,fill=white!100,thick, inner sep=0pt, minimum size= 0.5mm},
}

\newcommand{\WW}{\mathcal{W}}
\newcommand{\grid}{\mathbf{T}}
\usepackage{ulem}
\begin{document}

\maketitle

\begin{abstract}
We consider the 4-precoloring extension problem in \emph{planar near-Eulerian-
triangulations}, i.e., plane graphs where all faces except possibly for the outer one
have length three, all vertices not incident with the outer face have even degree,
and exactly the vertices incident with the outer face are precolored.  We give
a necessary topological condition for the precoloring to extend, and give a complete
characterization when the outer face has length at most five and when all vertices
of the outer face have odd degree and are colored using only three colors.
\end{abstract}

\section{Introduction}

Recall that a \emph{$k$-coloring} of a graph $G$ is a mapping using $k$ colors such that adjacent
vertices receive different colors and that a graph is \emph{Eulerian} if
all of its vertices have even degree. We study the precoloring
extension problem for planar (near) Eulerian triangulations from
an algorithmic perspective.

Famously, the Four Color Theorem states that all planar graphs are
$4$-colorable~\cite{fourcolthm} and thus from an algorithmic point of view,
the problem of determining if a planar graph is $4$-colorable is trivial.
Here, $4$-colorability cannot be improved to $3$-colorability, as deciding if a planar graph is $3$-colorable is a well known NP-complete problem~\cite{planar3col}.

Let us now consider the $k$-colorability for graphs on surfaces of non-zero genus $g$.  By Heawood's formula,
the problem becomes trivial in this setting for
$k=\Omega(g^{1/2})$, however it is actually quite well
understood for all $k\ge 5$.
Recall that a graph $G$ is \emph{$(k+1)$-critical} if
all proper subgraphs of $G$ are $k$-colorable, but $G$ itself is not.
Thus, $(k+1)$-critical graphs are exactly the minimal forbidden subgraphs for $k$-colorability.
A deep result of Thomassen~\cite{THOMASSEN} says that for any fixed surface $\Sigma$,
there are only finitely many $(k+1)$-critical graphs for $k \geq 5$.  Thus, we can test whether a graph drawn
on $\Sigma$ is $k$-colorable simply by checking whether
it contains one of this finite list of forbidden $(k+1)$-critical graphs as a subgraph (this can be done in
linear time by the result of Eppstein~\cite{subgraphs}). 

Unfortunately, we cannot extend this result to
$4$-colorablity---it is known that there are infinitely many
$5$-critical graphs on any surface other than the sphere~\cite{THOMASSEN}. This
is a consequence of the folowing result of Fisk~\cite{Fisk78}:
If $G$ is a triangulation of an orientable surface and $G$ has exactly
two vertices $u$ and $v$ of odd degree, then $u$ and $v$ must have the same color
in any 4-coloring of $G$, and thus the graph $G+uv$ is not 4-colorable.
Even though there are infinitely many $5$-critical graphs which embed on any surface of non-zero genus, it is an important open question if
for any fixed surface $\Sigma$, there is a polynomial time algorithm to decide if
a graph drawn on $\Sigma$ is $4$-colorable. 

The case of $3$-coloring triangle-free graphs on surfaces may shed some light on this problem. It is known that there are
infinitely many $4$-critical triangle-free graphs on all surfaces other than
the plane, yet there is a linear time algorithm to decide if a triangle-free
graph on any fixed surface is $3$-colorable~\cite{3colalgorithm}.
The algorithm consists of two parts: In the first part, the problem is
reduced to (near) quadrangulations~\cite{trfree4}, and the second part
gives a topological criterion for 3-colorability of near-quadrangulations~\cite{trfree6}.

Our hope is that (near) Eulerian triangulations could play the same
intermediate role in the case of 4-colorability of graphs on surfaces.
Indeed, there are a number of arguments and analogies supporting this idea:
\begin{itemize}
\item[(A)] The only constructions of ``generic'' (e.g., avoiding non-trivial small separations)
non-4-colorable graphs drawn on a fixed surface that we are aware of are based on
near Eulerian triangulations, such as Fisk's construction~\cite{Fisk78}
or adding vertices to faces of non-$3$-colorable quadrangulations of the projective plane~\cite{Mohar2002ColoringET}.
 
\item[(B)] As mentioned quadrangulations play a key role in the problem of 3-colorability of triangle-free
graphs on surfaces. Many results on coloring quadrangulations of surfaces correspond to results for Eulerian triangulations. For example:
\begin{itemize}
    \item Planar quadrangulations are $2$-colorable, and analogously, planar Eulerian triangulations are $3$-colorable, and in fact characterize $3$-colorability of planar graphs.
    \item Famously, Youngs~\cite{Youngs} proved that a graph drawn in the projective plane so that all faces have even
length is 3-colorable if and only if it does not contain a non-bipartite quadrangulation as a subgraph. Compare this to the fact that for Eulerian triangulation $T$ of the projective plane, Fisk~\cite{fisk} showed that $T$ has an independent set $U$ such that
all faces of $T-U$ have even length, and Mohar~\cite{Mohar2002ColoringET} proved that $T$ is 4-colorable if and only if $T-U$ is 3-colorable.
    \item Hutchinson~\cite{locplanq} proved that every graph drawn on a fixed orientable surface with only even-length faces
and with sufficiently large \emph{edge-width} (the length of the shortest non-contractible cycle) is 3-colorable,
and Nakamoto et al.~\cite{NakNegOta} and Mohar and Seymour~\cite{MohSey} have shown that such graphs on non-orientable surfaces are 3-colorable
unless they contain a quadrangulation with an odd-length orienting cycle.  Analogously,
for any orientable surface, any Eulerian triangulation with sufficiently large edge-width is 4-colorable~\cite{hutcheuler},
and for non-orientable surfaces, the only non-4-colorable Eulerian triangulations of large edge-width are those that
have an independent set whose removal results in an even-faced non-3-colorable graph~\cite{nakaeuler}
\end{itemize}

\end{itemize}

In this paper, we make the first step towards towards the design of a poly\-no\-mial-time algorithm
to decide whether an Eulerian triangulation of a fixed surface is $4$-colorable. In particular,
we give the following algorithm for \emph{planar near-Eulerian-triangulations}, i.e., plane graphs where all the faces except possibly for the outer one have length three and
all the vertices not incident with the outer face have even degree.
\begin{theorem}\label{precol3}
For every integer $\ell$, there is a linear-time algorithm that given
\begin{itemize}
\item a planar near-Eulerian-triangulation $G$ with the outer face bounded by a cycle $C$ of length at most $\ell$ such that all vertices of $C$ have \emph{odd degree} in $G$, and
\item a precoloring $\varphi$ of the vertices incident with the outer face of $G$ using \emph{only three colors},
\end{itemize}
decides whether $\varphi$ extends to a 4-coloring of $G$.
\end{theorem}

The motivation for considering the special case of planar near-Eulerian-triangulations with precolored outer face comes from the general approach
towards solving problems for graphs on surfaces, which can be seen e.g. in~\cite{trfree6,trfree2, trfree3, rs6,rs7,Th2,thomassen-surf}
and many other works and is explored systematically in the hyperbolic theory of Postle and Thomas~\cite{Hyperbolicfamilies}.
The general outline of this approach is as follows:
\begin{itemize}
\item Generalize the problem to surfaces with boundary, with the boundary vertices precolored (or otherwise constrained).
\item Use this generalization to reduce the problem to ``generic'' instances (e.g., those without short non-contractible
cycles, since if an instance contains a short non-contractible cycle, we can cut the surface along the cycle and try
to extend all the possible precolorings of the cycle in the resulting graph drawn in a simpler surface).
\item The problem is solved in the basic case of graphs drawn in a disk and on a cylinder
(plane graphs with one or two precolored faces).
\item Finally, the general case is solved with the help of the two basic cases by reducing it to the basic cases by further cutting
the surface and carefully selecting the constraints on the boundary vertices~\cite{trfree6, rs7}, or
using quantitative bounds from the basic cases to show that truly generic cases do not actually arise~\cite{Hyperbolicfamilies}.
\end{itemize}
Thus, Theorem~\ref{precol3} is a step towards solving the basic case of graphs drawn in a disk.
It unfortunately does not solve this case fully because of the extra assumption that $\varphi$ only uses three colors
(and to a lesser extent due to the assumption that vertices of $C$ have odd degree).
Without this extra assumption, we were able to solve the problem when the precolored outer face has length at most five.
\begin{theorem}\label{precolfive}
There is a linear-time algorithm that given
\begin{itemize}
\item a planar near-Eulerian-triangulation $G$ with the outer face bounded by a cycle $C$ of length at most five and
\item a precoloring $\varphi$ of $C$
\end{itemize}
decides whether $\varphi$ extends to a 4-coloring of $G$.
\end{theorem}
In general, we were only able to find a necessary topological condition for such an
extension to exist (Lemma~\ref{lemma-topol}).  Moreover, we show that this condition
is sufficient in the special case when $G$ has an independent set $U$ disjoint from $C$
such that all faces of $G-U$ except for the outer one have length four (Lemma~\ref{lemma-noshort}); the discussion in (B) above
justifies the importance of this special case.  We conjecture that using this topological condition
in combination with further ideas, it will be possible to resolve the disk case in full.
\begin{conjecture}
For every positive integer $\ell$, there is a polynomial-time algorithm that given
\begin{itemize}
\item a planar near-Eulerian-triangulation $G$ with the outer face of length at most $\ell$ and
\item a precoloring $\varphi$ of the vertices incident with the outer face
\end{itemize}
decides whether $\varphi$ extends to a 4-coloring of $G$.
\end{conjecture}
A step towards this conjecture would be to prove that the precoloring always extends
if $G$ is ``sufficiently generic''  as described in the following conjecture (for the definition of a viable precoloring, see Section~\ref{sec-homgrid},
where we show that if a precoloring of the outer face of a planar near-Eulerian-triangulation extends
to a 4-coloring of the whole graph, then it must be viable).
\begin{conjecture}
For every positive even integer $\ell$, there exists an integer $d$ such that the following claim holds.
Let $G$ be a planar near-Eulerian-triangulation with the outer face bounded by a cycle $C$ of length $\ell$
and let $U$ be an independent set disjoint from $C$ such that $G-U$ is bipartite. Moreover, suppose
that $G-U$ has a set $S$ of $d$ faces of length at least six such that
\begin{itemize}
\item there is no 4-cycle in $G-U$ separating a face of $S$ from the outer face, and
\item for distinct faces $s_1,s_2\in S$, the distance between $s_1$ and $s_2$ in $G-U$ is at least $d$
and there is no closed walk of length less than $\ell$ in $G-U$ separating both $s_1$ and $s_2$ from
the outer face.
\end{itemize}
Then any viable 4-coloring of $C$ extends to a 4-coloring of $G$.
\end{conjecture}

The structure of the paper and outline of the proofs are as follows. In Section~\ref{sec-homgrid},
we prove that the precoloring extension problem is equivalent to finding
color-preserving homomorphisms to the infinite triangular grid. In
Section~\ref{sec-hexa} we prove a common generalization of Theorems~\ref{precol3} and \ref{precolfive}
by applying an old result of Hell~\cite{hellretraction} to the triangular grid.
In Section~\ref{sec-homot}, we describe the necessary topological condition 
for precoloring extension in planar near-Eulerian-triangulations, and show that this condition is
sufficient in the case where the near-Eulerian-triangulation arises from a near-quadrangulation as described above.

\section{Homomorphisms to triangular grid}\label{sec-homgrid}

A \emph{patch} is a connected plane graph with all faces of length three except possibly for
the outer face.  A vertex of a patch is \emph{internal} if it is not incident with the outer face.
Let us note the following standard observation.
\begin{observation}\label{obs-3coreu}
A patch $G$ is 3-colorable if and only if all internal vertices of $G$ have even degree,
i.e., if $G$ is a near-Eulerian-triangulation.
\end{observation}
\begin{proof}
Without loss of generality, we can assume that $G$ is 2-connected, as otherwise
we can $3$-color each 2-connected block of $G$ separately and combine the colorings  (after possibly permuting the colors) to obtain a $3$-coloring of $G$. It follows that the outer face of $G$ is bounded by a cycle. Let $G'$ be the plane triangulation obtained by taking two copies of $G$ and identifying
the corresponding vertices on the outer face.  Clearly $G$ is $3$-colorable if and only if $G'$ is
3-colorable, and it is well-known that a plane triangulation is $3$-colorable if and only if it is
Eulerian.
\end{proof}
Note that if the patch is $2$-connected, then its 3-coloring is unique up to a permutation
of colors, but otherwise there is some ambiguity in the choice of the 3-coloring.
Motivated by this observation, it will be convenient to work with graphs equipped with a fixed 3-coloring.
A \emph{hued graph} is a graph $G$ together with a proper 3-coloring $\psi_G:V(G)\to \mathbb{Z}_3$.  

It will also be often useful to additionally fix a 4-coloring of $G$, and it will be notationally convenient to use the
elements of $\mathbb{Z}_2^2$ as colors.  Hence, a \emph{dappled graph} is a hued graph $G$ together with 
a proper 4-coloring  $\varphi_G:V(G)\to \mathbb{Z}_2^2$.  For a vertex $v\in V(G)$, we say 
that $\psi_G(v)$ is the \emph{hue} and $\varphi_G(v)$ is the \emph{color} of $v$. 
For a hued graph $H$ and a 4-coloring $\theta$ of $H$, let $H^\theta$ denote the corresponding
dappled graph such that $\varphi_{H^\theta}=\theta$.  Conversely, for a dappled graph $G$, let $G^{-}$
denote the hued graph obtained by forgetting the colors of vertices.

\emph{Homomorphisms} of dappled graphs are required to preserve edges and both hue and color, i.e., 
$f:V(G)\to V(H)$ is a homomorphism if $f(u)f(v)\in E(H)$ for every $uv\in E(G)$ and
$\psi_H(f(v))=\psi_G(v)$ and $\varphi_H(f(v))=\varphi_G(v)$ for every $v\in V(G)$.

The \emph{dappled triangular grid} (see Figure~\ref{fig-grid}) is the infinite dappled graph $\grid$ with
\begin{itemize}
\item vertex set $\{(i,j):i,j\in\mathbb{Z}\}$,
\item vertices $(i_1,j_1)$ and $(i_2,j_2)$ adjacent if and only if 
$(i_2-i_1,j_2-j_1)\in\{\pm(1,0),\pm(0,1),\pm(1,1)\}$,
\item hue $\psi_\grid(i,j)=(i+j)\bmod 3$ for each vertex $(i,j)$, and
\item color $\varphi_\grid(i,j)=(i\bmod 2,j\bmod 2)$ for each vertex $(i,j)$.
\end{itemize}

\begin{figure}
\begin{center}
\begin{tikzpicture}
\node[greenvertex] at (-2,1) (v-2) {};
\node[bluevertexv2] at (-1,1) (v-1) {};
\node[greenvertexv3] at (0,1) (v0) {};
\node[bluevertex] at (1,1) (v1) {};
\node[greenvertexv2] at (2,1) (v2) {};
\node[bluevertexv3] at (3,1) (v3) {};
\node[greenvertex] at (4,1) (v4) {};
\node[bluevertexv2] at (5,1) (v5) {};

\node[yellowvertexv3] at (-2,0) (u-2) {};
\node[redvertex] at (-1,0) (u-1) {};
\node[yellowvertexv2] at (0,0) (u0) {};
\node[redvertexv3] at (1,0) (u1) {};
\node[yellowvertex] at (2,0) (u2) {};
\node[redvertexv2] at (3,0) (u3) {};
\node[yellowvertexv3] at (4,0) (u4) {};
\node[redvertex] at (5,0) (u5) {};

\node[greenvertexv2] at (-2,-1) (z-2) {};
\node[bluevertexv3] at (-1,-1) (z-1) {};
\node[greenvertex] at (0,-1) (z0) {};
\node[bluevertexv2] at (1,-1) (z1) {};
\node[greenvertexv3] at (2,-1) (z2) {};
\node[bluevertex] at (3,-1) (z3) {};
\node[greenvertexv2] at (4,-1) (z4) {};
\node[bluevertexv3] at (5,-1) (z5) {};


\node[yellowvertex] at (-2,-2) (t-2) {};
\node[redvertexv2] at (-1,-2) (t-1) {};
\node[yellowvertexv3] at (0,-2) (t0) {};
\node[redvertex] at (1,-2) (t1) {};
\node[yellowvertexv2] at (2,-2) (t2) {};
\node[redvertexv3] at (3,-2) (t3) {};
\node[yellowvertex] at (4,-2) (t4) {};
\node[redvertexv2] at (5,-2) (t5) {};


\node[greenvertexv3] at (-2,-3) (r-2) {};
\node[bluevertex] at (-1,-3) (r-1) {};
\node[greenvertexv2] at (0,-3) (r0) {};
\node[bluevertexv3] at (1,-3) (r1) {};
\node[greenvertex] at (2,-3) (r2) {};
\node[bluevertexv2] at (3,-3) (r3) {};
\node[greenvertexv3] at (4,-3) (r4) {};
\node[bluevertex] at (5,-3) (r5) {};

\draw[thick,black] (u-2)--(u-1)--(u0)--(u1)--(u2)--(u3)--(u4)--(u5);
\draw[thick,black] (v-2)--(v-1)--(v0)--(v1)--(v2)--(v3)--(v4)--(v5);
\draw[thick,black] (z-2)--(z-1)--(z0)--(z1)--(z2)--(z3)--(z4)--(z5);
\draw[thick,black] (t-2)--(t-1)--(t0)--(t1)--(t2)--(t3)--(t4)--(t5);
\draw[thick,black] (r-2)--(r-1)--(r0)--(r1)--(r2)--(r3)--(r4)--(r5);

\draw[thick,black] (v-2)--(u-2)--(z-2)--(t-2)--(r-2);
\draw[thick,black] (v-1)--(u-1)--(z-1)--(t-1)--(r-1);
\draw[thick,black] (v0)--(u0)--(z0)--(t0)--(r0);
\draw[thick,black] (v1)--(u1)--(z1)--(t1)--(r1);
\draw[thick,black] (v2)--(u2)--(z2)--(t2)--(r2);
\draw[thick,black] (v3)--(u3)--(z3)--(t3)--(r3);
\draw[thick,black] (v4)--(u4)--(z4)--(t4)--(r4);
\draw[thick,black] (v5)--(u5)--(z5)--(t5)--(r5);

\draw[thick,black] (v-1)--(u-2);
\draw[thick,black] (v0)--(u-1)--(z-2);
\draw[thick,black] (v1)--(u0)--(z-1)--(t-2);
\draw[thick,black] (v2)--(u1)--(z0)--(t-1)--(r-2);
\draw[thick,black] (v3)--(u2)--(z1)--(t0)--(r-1);
\draw[thick,black] (v4)--(u3)--(z2)--(t1)--(r0);
\draw[thick,black] (v5)--(u4)--(z3)--(t2)--(r1);
\draw[thick,black] (u5)--(z4)--(t3)--(r2);
\draw[thick,black] (z5)--(t4)--(r3);
\draw[thick,black] (t5)--(r4);

\draw[thick,dashed,blue] (r0)--(t2)--(z4)--(u3)--(v2)--(u0)--(z-2)--(t-1)--(r0);

\end{tikzpicture}
\caption{A portion of the dappled triangular grid~$\grid$. The hues of the vertices are represented by their shapes.
The dashed line indicates toroidal drawing of the graph $C_3\times C_4$ for which $\grid$ is the universal cover.}
\label{fig-grid}
\end{center}
\end{figure}
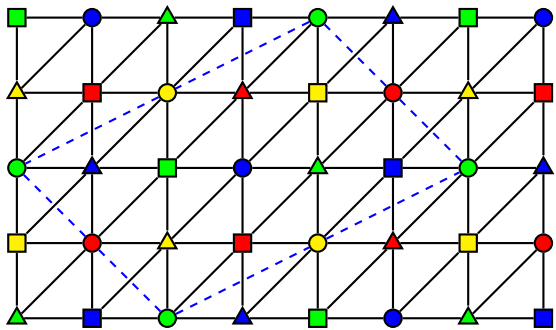

Let us remark that since $\grid$ is a triangulation, the hue is unique up to permutation
of colors, while the choice of the 4-coloring is more deliberate.  In particular,
the following claim follows by an inspection of the definition.
\begin{observation}\label{obs-homtot}
Let $\grid$ be the dappled triangular grid.  For any vertex $v\in V(\grid)$,
if $u_1$ and $u_2$ are distinct neighbors of $v$, then
$(\psi_\grid(u_1),\varphi_\grid(u_1))\neq(\psi_\grid(u_2),\varphi_\grid(u_2))$.
Consequently, for any connected dappled graph $G$ and any vertex $x\in V(G)$,
if $f_1,f_2:V(G)\to V(\grid)$ are homomorphisms and $f_1(x)=f_2(x)$, then $f_1=f_{2}$.
\end{observation}

Let us now give the key property of dappled patches; the proof easily follows from the coloring-flow
duality of Tutte~\cite{tutteduals}, we give the details for completeness.
\begin{theorem}\label{thm-main}
Every dappled patch $G$ has a homomorphism to the dappled triangular grid $\grid$.
\end{theorem}
\begin{proof}
For $a\in\{1,2\}$ and $(b,c)\in \{(0,1),(1,0),(1,1)\}$, let
\begin{align*}
\sigma(a,(b,c))&=\begin{cases}
1&\text{if $a=b+c$}\\
-1&\text{otherwise}
\end{cases}\\
\delta(a,(b,c))&=\sigma(a,(b,c))\cdot (b,c).
\end{align*}
For each pair $u,v\in V(G)$ of adjacent vertices of $G$, let us define
$$\delta(u,v)=\delta(\psi_G(v)-\psi_G(u), \varphi_G(v)-\varphi_G(u)).$$
Here, the subtractions are performed in $\mathbb{Z}_3$ and in $\mathbb{Z}_2^2$, respectively,
and the results are interpreted as integers in $\{1,2\}$ or pairs of integers in $\{(0,1),(1,0),(1,1)\}$.
Note that it is not possible for the differences to be $0$ or $(0,0)$, since $\psi_G$ and $\varphi_G$
are proper colorings.  Clearly $\delta(u,v)=-\delta(v,u)$.  For a walk $W=v_0v_1\ldots v_k$ in $G$, let
$$\delta(W)=\sum_{i=1}^k\delta(v_{i-1},v_i).$$
Observe that for any triangle $Q=v_0v_1v_2v_3$ in $G$ with $v_0=v_3$, there exists $a\in\{1,2\}$
such that 
$$\psi_G(v_i)-\psi_G(v_{i-1})=a\text{ for all $i\in\{1,2,3\}$},$$
and
$$\{\varphi_G(v_i)-\varphi_G(v_{i-1}):i\in\{1,2,3\}\}=\{(0,1),(1,0),(1,1)\}$$
Consequently,
$$\delta(Q)=\delta(a,(0,1))+\delta(a,(1,0))+\delta(a,(1,1))=(0,0).$$
Since this holds for the boundary of every internal face of $G$, we conclude that also for
every closed walk $C$ in $G$, we have $\delta(C)=0$.  Therefore, for any (not necessarily adjacent)
vertices $u,v\in V(G)$, we can define $\delta(u,v)=\delta(W)$ for any walk $W$ from $u$ to $v$,
and the definition does not depend on the choice of $W$.

Let $u$ be an arbitrary vertex of $G$ and let $u'$ be a vertex of $\grid$ such that 
$\psi_G(u)=\psi_\grid(u')$ and $\varphi_G(u)=\varphi_\grid(u')$; such a vertex $u'$ exists, since
$\grid$ contains vertices with all combinations of hue and color.  Let $f:V(G)\to V(\grid)$ be defined
by letting $f(v)=u'+\delta(u,v)$ for each $v\in V(G)$.  It is straightforward to verify that
$f$ is a homomorphism from $G$ to $\grid$:
\begin{itemize}
\item If $xy\in E(G)$, then $f(y)-f(x)=\delta(x,y)$, as seen by considering
a walk $W$ from $u$ to $x$ and the walk from $u$ to $y$ obtained by appending the edge $xy$ at the end of $W$.
Consequently $f(y)-f(x)\in\{\pm(0,1),\pm(1,0),\pm(1,1)\}$, and thus $f(x)$ and $f(y)$ are adjacent in $\grid$.
\item We prove that each vertex $y\in V(G)$ is mapped to a vertex of the same hue and color in $\grid$ by induction
on the distance between $y$ and $u$ in $G$.  If $y=u$, then the claim holds by the choice of $u'$.
Hence, suppose that $y\neq u$ and let $x$ be a neighbor of $y$ on a shortest path from $u$ to $y$ in $G$.
By the induction hypothesis, $\psi_\grid(f(x))=\psi_G(x)$ and $\varphi_\grid(f(x))=\varphi_G(x)$.
Let $a=\psi_G(y)-\psi_G(x)$ and $(b,c)=\varphi_G(y)-\varphi_G(x)$.  Recall that $f(y)-f(x)=\delta(x,y)=\sigma(a,(b,c))\cdot (b,c)$.
By the definition of the hue and coloring of $\grid$, when computing in $\mathbb{Z}_3$, we have
$$\psi_\grid(f(y))-\psi_\grid(f(x))=\sigma(a,(b,c))\cdot (b+c)=\begin{lrcases}
b+c&\text{if $a=b+c$}\\
-(b+c)&\text{if $a=-(b+c)$}
\end{lrcases}=a,$$
and when computing in $\mathbb{Z}_2^2$, we have
$$\varphi_\grid(f(y))-\varphi_\grid(f(x))=(b,c).$$
Therefore,
\begin{align*}
\psi_\grid(f(y))&=\psi_\grid(f(x))+a=\psi_G(x)+a=\psi_G(y)\\
\varphi_\grid(f(y))&=\varphi_\grid(f(x))+(b,c)=\varphi_G(x)+(b,c)=\varphi_G(y).
\end{align*}
\end{itemize}
\end{proof}
Let us remark that another way of interpreting (and proving) Theorem~\ref{thm-main} is as follows: The hue and coloring of $G$ determine a homomorphism from $G$ to the graph $K_{\mathbb{Z}_3}\times K_{\mathbb{Z}_2^2}$, i.e. the categorical product of cliques with $3$ and $4$ vertices\footnote{For graphs $G_1$ and $G_2$,
$G_1\times G_2$ is the graph with vertex set
$V(G_1)\times V(G_2)$ and with $(u_1,u_2)$ adjacent to $(v_1,v_2)$ if and only if $u_1v_1\in E(G_1)$ and $u_2v_2\in E(G_2)$.}.  This graph can be drawn as a triangulation
of the torus, and $\grid$ is its universal cover.

A $4$-coloring $\varphi$ of a connected hued graph $C$ is \emph{viable} if and only if the dappled graph $C^\varphi$ has a homomorphism to the dappled triangular grid.
Let us note that by Observation~\ref{obs-homtot}, this condition is easy to verify, as
such a homomorphism, if it exists, can be found in linear time by extending a partial homomorphism step-by-step. 

\begin{corollary}
Let $G$ be a hued patch and let $\varphi$ be a $4$-coloring of the boundary of the outer face of $G$.
If $\varphi$ extends to a $4$-coloring of $G$, then $\varphi$ is viable.
\end{corollary}
This condition is topological in nature; another way how to view it is as follows:  Recall that
the dappled triangular grid $\grid$ is the universal cover of the dappled graph $K_{\mathbb{Z}_3}\times K_{\mathbb{Z}_2^2}$ drawn on the torus.
A 4-coloring of a hued cycle $C$ is viable if and only if the corresponding homomorphism to $K_{\mathbb{Z}_3}\times K_{\mathbb{Z}_2^2}$
maps $C$ to a contractible closed walk on the torus.

Note that viability is a necessary but of course not sufficient condition.  However, it
is the strictest possible condition based only on the coloring and hue of the boundary of
the outer face, in the following sense.

\begin{lemma}\label{lemma-exists}
If $C$ is a dappled cycle with $\varphi_C$ viable, then there exists a dappled patch $G$ with
the outer face bounded by a dappled cycle isomorphic to $C$.
\end{lemma}
\begin{proof}
Let us prove the claim by induction on the length of $C$.  Let $f:V(C)\to V(\grid)$ be a homomorphism
from $C$ to the dappled triangular grid $\grid$ that exists by the viability of $\varphi_C$.
If $f$ is injective, then we can let $G$ be the subgraph of $\grid$ drawn inside the cycle $f(C)$.

Otherwise, there exist distinct vertices $u,v\in V(C)$ such that $f(u)=f(v)$.  Since $f$ is a homomorphism,
this implies that $u$ and $v$ have the same color and hue, and in particular $u$ and $v$ are not adjacent in $C$.
Let $C_1$ and $C_2$ be the 2-connected blocks of the dappled graph obtained from $C$ by
identifying $u$ with $v$ to a single vertex $z$ ($C_1$ and $C_2$ are cycles if the distance between $u$ and $v$ in $C$ is at least three,
and otherwise at least one of them consists of a single edge).  For $i\in \{1,2\}$, the
restriction of $f$ to $C_i$ is a homomorphism to $\grid$, and thus $\varphi_{C_i}$ is viable.
If $C_i$ is just an edge, then let $G_i=C_i$, otherwise let $G_i$ be a dappled patch with
the outer face bounded by a dappled cycle isomorphic to $C_i$, which exists by the induction
hypothesis.  

Let $G'$ be the dappled patch obtained from the disjoint union of $G_1$ and $G_2$
by identifying the vertices corresponding to $z$.  Let $xzy$ be a part of the boundary walk of
the outer face of $G'$.  The dappled patch $G$ is obtained from $G'$ by the following operation,
ensuring that the outer face of $G$ is bounded by a cycle isomorphic to $C$ (see Figure~\ref{fig-combine} for an illustration):
Add a path $Q$ between $x$ and $y$ of length two if $x$ and $y$ have the same hue and of length three
if they have a different hue.  Add a vertex $w$, and edges between $w$ and all vertices of $Q$,
and between $z$ and all vertices of $Q$.  Observe that both $\psi_{G'}$ and $\varphi_{G'}$
can be extended to $G$ so that $w$ has the same hue and color as $z$.
\end{proof}

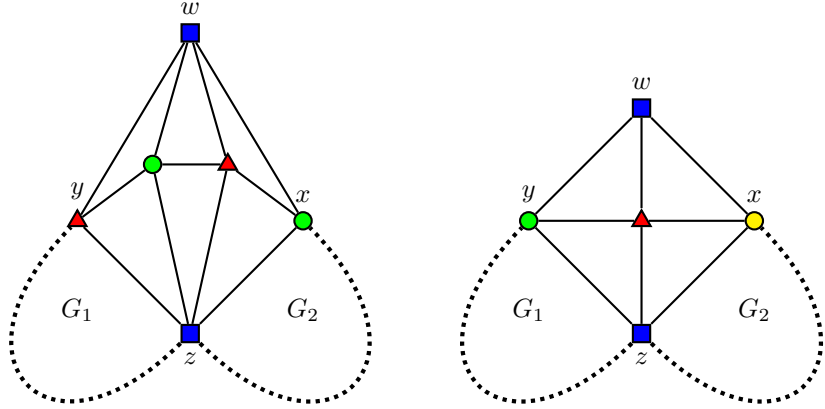
\begin{figure}
\begin{center}
\begin{tikzpicture}
\node[bluevertex] at (0,.5) (z) [label=below:$z$] {};
\node[greenvertexv2] at (1.5,2) (x) [label=above:$x$] {};
\node[redvertexv3] at (-1.5,2) (y) [label=above:$y$] {};
\node[redvertexv3] at (.5,2.75) (x1) {};
\node[greenvertexv2] at (-.5,2.75) (x2) {};
\node[bluevertex] at (0,4.5) (w) [label =above:$w$] {};
\draw[thick,black] (y)--(z)--(x)--(x1)--(x2)--(y);
\draw[thick,black] (z)--(x1);
\draw[thick,black] (z)--(x2);
\draw[thick,black] (w)--(y);
\draw[thick,black] (w)--(x1);
\draw[thick,black] (w)--(x2);
\draw[thick,black] (w)--(x);
\draw[ultra thick,black, dotted, out=225, in=225, looseness = 3] (z) to (y);
\draw[ultra thick,black, dotted, out=-45, in=-45, looseness = 3] (z) to (x);
\node[dummywhite] at (-1.5,.5) (dummy1) [label=above:$G_{1}$] {};
\node[dummywhite] at (1.5,.5) (dummy2) [label=above:$G_{2}$] {};

\begin{scope}[xshift =6cm]
\node[bluevertex] at (0,.5) (z) [label=below:$z$] {};
\node[yellowvertexv2] at (1.5,2) (x) [label=above:$x$] {};
\node[greenvertexv2] at (-1.5,2) (y) [label=above:$y$] {};
\node[redvertexv3] at (0,2) (x1) {};
\node[bluevertex] at (0,3.5) (w) [label =above:$w$] {};
\draw[thick,black] (y)--(z)--(x)--(x1)--(y);
\draw[thick,black] (z)--(x1);
\draw[thick,black] (w)--(y);
\draw[thick,black] (w)--(x1);
\draw[thick,black] (w)--(x);
\draw[ultra thick,black, dotted, out=225, in=225, looseness = 3] (z) to (y);
\draw[ultra thick,black, dotted, out=-45, in=-45, looseness = 3] (z) to (x);
\node[dummywhite] at (-1.5,.5) (dummy1) [label=above:$G_{1}$] {};
\node[dummywhite] at (1.5,.5) (dummy2) [label=above:$G_{2}$] {};
\end{scope}
\end{tikzpicture}
\end{center}
\caption{Construction from the proof of Lemma~\ref{lemma-exists}.  The vertex shapes indicate their hue.
The color choices represent one possible extension, depending on if the hue's of $y$ and $x$ are the same or not.}\label{fig-combine}
\end{figure}

\section{Single-hexagon precolorings}\label{sec-hexa}

In this section we aim to prove a common generalization of Theorems~\ref{precol3} and \ref{precolfive}.
The key observation is that in both cases, the homomorphism to the dappled triangular grid $\grid$
associated with the precoloring $\varphi$ maps $C$ to the neighborhood of a single vertex of $\grid$.

A \emph{hexagon} is a dappled subgraph of $\grid$ induced by a vertex
and its neighbors. A $4$-coloring $\varphi$ of a connected hued graph $C$ is
a \emph{single-hexagon coloring} if it is viable and the corresponding
homomorphism $f$ maps $C^\varphi$ to a subgraph of a hexagon $H$ of $\grid$.
The \emph{central hue} $c$ and the \emph{central color} $k$ of a single-hexagon
coloring is the hue and the color of the central vertex of $H$. Let us remark that a vertex of $C$ has color $k$
if and only if it has hue $c$.  There are two important examples of single-hexagon colorings,
corresponding to the assumptions of Theorems~\ref{precol3} and \ref{precolfive}, respectively.
\begin{itemize}
    \item Let us call a patch (recall a patch is a planar graph where all internal faces are triangles) \emph{odd} if its outer face is bounded by a cycle $C$ and all vertices incident 
     with the outer face have odd degree.  Note that in an odd hued patch, all vertices of $C$
     have one of two values of hue, say $0$ and $1$.  Every coloring of $C$ that uses at most three colors
     is single-hexagon, with the central hue $2$ and central color not appearing on $C$.
    \item Every viable $4$-coloring of a hued $(\le\!5)$-cycle is single-hexagon.
\end{itemize}

A \emph{retract} of a hued graph $G$ is an induced subgraph $H$ of $G$ such that there exists a 
\emph{retraction} $f$ from $G$ to $H$, i.e., a homomorphism such that $f(v)=v$ for each $v\in V(H)$.
Let us note the following key observation, which is a consequence of the fact that each shortest
cycle in a bipartite graph is a retract~\cite{hellretraction}; we include a proof for completeness. See Figure \ref{fig-retract} for a picture of the retraction. Recall that for a dappled graph $G$, the graph $G^{-}$ is obtained from $G$ by forgetting the colors of the vertices.
\begin{lemma}\label{lem-retract}
For every hexagon $X$ of the triangular grid $\grid$, the hued hexagon $X^{-}$ is a retract of $\grid^-$.
\end{lemma}
\begin{proof}
By symmetry, we can assume that $X$ is the subgraph of $\grid$ induced by the vertex $(1,1)$
and its neighbors.  Let $H$ be the subgraph of $\grid$ induced by vertices of hue $0$ and $1$,
and let $e$ be the edge of $H$ between vertices $(0,0)$ and $(0,1)$.  Let
$h:\mathbb{Z}_0^+\to V(X)$ be defined as follows:
$$h(m)=\begin{cases}
(0,0)&\text{if $m=0$}\\
(1,0)&\text{if $m=1$}\\
(2,1)&\text{if $m=2$}\\
(2,2)&\text{if $m=3$}\\
(1,2)&\text{if $m\ge 4$ is even}\\
(0,1)&\text{if $m\ge 5$ is odd}
\end{cases}$$
Observe that for each $m\le 5$, the distance of $h(m)$ from $(0,0)$ in $H-e$ is exactly $m$.
Moreover, note that $H-e$ is bipartite and the vertices at even distance from $(0,0)$ have hue $0$
and those at odd distance have hue $1$. For each vertex $v$ of $\grid$, define $f(v)=(1,1)$ if
$v$ has hue $2$ and $f(v)=h(m)$ if $v$ has hue $0$ or $1$ and its distance from $(0,0)$ in
$H-e$ is $m$.  Clearly, $f$ preserves hue and it is the identity on $X$.

Consider any edge $uv$ of $G$.
\begin{itemize}
\item If $uv\in E(H-e)$, then since $H-e$ is bipartite, the distances from $(0,0)$ to $u$ and $v$ in $H-e$
differ by exactly one.  Observe that $h(0)h(1)h(2)\ldots$ is a walk in $X$, and consequenly
$f$ maps $u$ and $v$ to adjacent vertices of $X$.
\item If $uv=e$, then $f$ maps $u$ and $v$ to adjacent vertices $h(0)=(0,0)$ and $h(5)=(0,1)$ of $X$.
\item If $uv\not\in E(H)$, then exactly one of $u$ and $v$ (say $u$) has hue $2$, and thus
$f$ maps $u$ to the central vertex of $X$ and $v$ to a different vertex of $X$.
\end{itemize}
Therefore $f$ preserves the edges, and consequently $f$ is a retraction from $\grid^-$ to~$X^-$.
\end{proof}

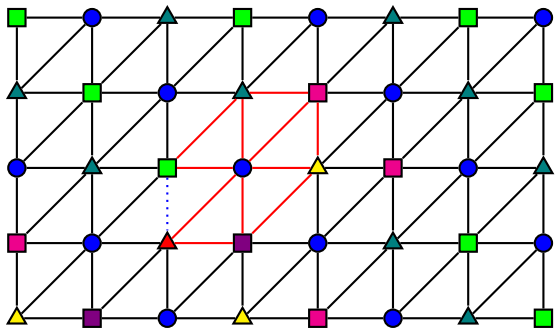
\begin{figure}
\begin{center}
\begin{tikzpicture}
\node[greenvertex] at (-2,1) (v-2) {};
\node[bluevertexv2] at (-1,1) (v-1) {};
\node[tealvertexv3] at (0,1) (v0) {};
\node[greenvertex] at (1,1) (v1) {};
\node[bluevertexv2] at (2,1) (v2) {};
\node[tealvertexv3] at (3,1) (v3) {};
\node[greenvertex] at (4,1) (v4) {};
\node[bluevertexv2] at (5,1) (v5) {};

\node[tealvertexv3] at (-2,0) (u-2) {};
\node[greenvertex] at (-1,0) (u-1) {};
\node[bluevertexv2] at (0,0) (u0) {};
\node[tealvertexv3] at (1,0) (u1) {};
\node[magentavertex] at (2,0) (u2) {};
\node[bluevertexv2] at (3,0) (u3) {};
\node[tealvertexv3] at (4,0) (u4) {};
\node[greenvertex] at (5,0) (u5) {};

\node[bluevertexv2] at (-2,-1) (z-2) {};
\node[tealvertexv3] at (-1,-1) (z-1) {};
\node[greenvertex] at (0,-1) (z0) {};
\node[bluevertexv2] at (1,-1) (z1) {};
\node[yellowvertexv3] at (2,-1) (z2) {};
\node[magentavertex] at (3,-1) (z3) {};
\node[bluevertexv2] at (4,-1) (z4) {};
\node[tealvertexv3] at (5,-1) (z5) {};

\node[magentavertex] at (-2,-2) (t-2) {};
\node[bluevertexv2] at (-1,-2) (t-1) {};
\node[redvertexv3] at (0,-2) (t0) {};
\node[violetvertex] at (1,-2) (t1) {};
\node[bluevertexv2] at (2,-2) (t2) {};
\node[tealvertexv3] at (3,-2) (t3) {};
\node[greenvertex] at (4,-2) (t4) {};
\node[bluevertexv2] at (5,-2) (t5) {};


\node[yellowvertexv3] at (-2,-3) (r-2) {};
\node[violetvertex] at (-1,-3) (r-1) {};
\node[bluevertexv2] at (0,-3) (r0) {};
\node[yellowvertexv3] at (1,-3) (r1) {};
\node[magentavertex] at (2,-3) (r2) {};
\node[bluevertexv2] at (3,-3) (r3) {};
\node[tealvertexv3] at (4,-3) (r4) {};
\node[greenvertex] at (5,-3) (r5) {};

\draw[thick,red] (u1)--(u2);
\draw[thick,red] (u1)--(z1);
\draw[thick,red] (u2)--(z2);
\draw[thick,red] (z1)--(z2);
\draw[thick,red] (u2)--(z1);
\draw[thick,red] (z1)--(z0);
\draw[thick,red] (z0)--(u1);
\draw[thick,blue,dotted] (z0) to (t0);
\draw[thick,red] (t1)--(t0);
\draw[thick,red] (t0)--(z1);
\draw[thick,red] (t1)--(z2);
\draw[thick,red] (t1)--(z1);

\draw[thick,black] (u-2)--(u-1)--(u0)--(u1);
\draw[thick,black] (u2)--(u3)--(u4)--(u5);
\draw[thick,black] (v-2)--(v-1)--(v0)--(v1)--(v2)--(v3)--(v4)--(v5);
\draw[thick,black] (z-2)--(z-1)--(z0);
\draw[thick,black] (z2)--(z3)--(z4)--(z5);
\draw[thick,black] (t-2)--(t-1)--(t0);
\draw[thick,black](t1)--(t2)--(t3)--(t4)--(t5);
\draw[thick,black] (r-2)--(r-1)--(r0)--(r1)--(r2)--(r3)--(r4)--(r5);

\draw[thick,black] (v-2)--(u-2)--(z-2)--(t-2)--(r-2);
\draw[thick,black] (v-1)--(u-1)--(z-1)--(t-1)--(r-1);
\draw[thick,black] (v0)--(u0)--(z0);
\draw[thick,black] (t0)--(r0);
\draw[thick,black] (v1)--(u1);
\draw[thick,black] (t1)--(r1);
\draw[thick,black] (v2)--(u2);
\draw[thick,black] (z2)--(t2)--(r2);
\draw[thick,black] (v3)--(u3)--(z3)--(t3)--(r3);
\draw[thick,black] (v4)--(u4)--(z4)--(t4)--(r4);
\draw[thick,black] (v5)--(u5)--(z5)--(t5)--(r5);

\draw[thick,black] (v-1)--(u-2);
\draw[thick,black] (v0)--(u-1)--(z-2);
\draw[thick,black] (v1)--(u0)--(z-1)--(t-2);
\draw[thick,black] (v2)--(u1);
\draw[thick,black] (z0)--(t-1)--(r-2);
\draw[thick,black] (v3)--(u2);
\draw[thick,black] (t0)--(r-1);
\draw[thick,black] (v4)--(u3)--(z2);
\draw[thick,black](t1)--(r0);
\draw[thick,black] (v5)--(u4)--(z3)--(t2)--(r1);
\draw[thick,black] (u5)--(z4)--(t3)--(r2);
\draw[thick,black] (z5)--(t4)--(r3);
\draw[thick,black] (t5)--(r4);

\end{tikzpicture}
\caption{A picture showing the retraction for Lemma~\ref{lem-retract}. The hexagon with red edges is target hexagon, with the blue circle in the red hexagon being the node $(1,1)$. The blue dotted edge indicates the edge $(0,0)(0,1)$ which is the edge deleted in the proof of Lemma~\ref{lem-retract}. The shapes on the vertex nodes indicate the respective hues, and colors represent which vertex in the retraction the vertices get mapped to. The graph $H$ is the graph obtained by deleting all blue circles.}
\label{fig-retract}

\end{center}
\end{figure}

This implies that the precoloring extension problem in hued patches for single-hexagon precolorings
can be reduced to the problem of 3-precoloring extension in planar bipartite graphs.

\begin{corollary}\label{cor-eq3}
Let $G$ be a hued patch and let $\varphi:V(C)\to\mathbb{Z}_2^2$ be a single-hexagon 4-coloring of
the boundary $C$ of the outer face of $G$, with central hue $c$ and central color $k$.
Let $K=\mathbb{Z}_2^2\setminus \{k\}$.  Let $H$ be the bipartite subgraph of 
$G$ induced by the vertices of hue different from $c$, and let $\varphi'$ be the 
restriction of $\varphi$ to $H$.  Then $\varphi$ extends to a 4-coloring of $G$ if and only
if $\varphi'$ extends to a 3-coloring of $H$ (using the colors in $K$).
\end{corollary}
\begin{proof}
The vertices of $G$ of hue $c$ form an independent set; hence, a 3-coloring of $H$ extending
$\varphi'$ can be extended to a 4-coloring $\theta$ of $G$ by giving all vertices of hue $c$ the 
color $k$. Recall that all vertices of $C$ of hue $c$ have color $k$, and thus $\theta$ extends
$\varphi$.

Conversely, suppose that $\varphi$ extends to a 4-coloring $\beta$ of $G$.  By
Theorem~\ref{thm-main}, the dappled graph $G^\beta$ has a homomorphism $f$ to
the dappled triangular grid $\grid$. Since $\varphi$ is a single-hexagon 4-coloring, there exists
a hexagon $X$ in $\grid$ such that the restriction of $f$ to $C$ is a homomorphism to $X$.
By Lemma~\ref{lem-retract}, there exists a retraction $r$ from $\grid^-$ to $X^-$.
Let $f'$ be the composition of $f$ and $r$.  Then $f'$ is a homomorphism
from $G$ to $X^-$ such that $f'(v)=f(v)$ for every $v\in V(C)$.
Let $\beta':V(G)\to \mathbb{Z}_2^2$ be defined by setting $\beta'(u)=\varphi_X(f'(u))$ for each $u\in V(G)$.
Since $f'$ is a homomorphism from $G$ to $X^-$ and $\varphi_X$ is a proper 4-coloring of $X^-$,
$\beta'$ is a proper 4-coloring of $G$.  Moreover, for each $v\in V(C)$,
we have $\beta'(v)=\varphi_X(f'(v))=\varphi_\grid(f(v))=\beta(v)=\varphi(v)$,
and thus $\beta'$ extends $\varphi$ and is a single-hexagon colouring. 

Every vertex of $X$ other than the central one has hue different from $c$ and color different from $k$,
and thus $\beta'$ assigns color $k$ exactly to vertices of $G$ of hue $c$.
Consequently, the restriction of $\beta'$ to $H$ is a 3-coloring of $H$ using the
colors in $K$ and extending $\varphi'$.
\end{proof}

Dvořák, Král' and Thomas~\cite{3colalgorithm} gave for every $b$ a linear-time algorithm to decide whether
a precoloring of at most $b$ vertices of a planar triangle-free graph extends to a 3-coloring.
Hence, we have the following consequence, a common strengthening of Theorems~\ref{precol3} and \ref{precolfive}.
\begin{corollary}\label{cor-algsing}
For every integer $\ell$, there exists an algorithm that, given a hued patch $G$ with
the outer face of length at most $\ell$ and a single-hexagon precoloring $\varphi$ of the boundary of
the outer face, decides in linear time whether $\varphi$ extends to a 4-coloring of $G$.
\end{corollary}

Corollary~\ref{cor-eq3} shows that the problem of extending the precoloring of the outer face
in Eulerian triangulations generalizes the analogous problem in bipartite graphs.
For a 2-connected bipartite plane graph $B$, the \emph{patch extension} of $B$ is the hued odd patch
created by adding a vertex to each face except for the outer one and joining it with all vertices
in the boundary of the face; the vertices of $B$ are given hues $0$ and $1$ and
the newly added vertices are given hue $2$.

\begin{corollary}
Let $B$ be a 2-connected bipartite planar graph.  A $3$-coloring of the boundary of the outer
face of $B$ extends to a 3-coloring of $B$ if and only if it extends to a $4$-coloring of
the patch extension of $B$.
\end{corollary}

\section{Homotopy}\label{sec-homot}

We were not able to find an exact characterization for extendability
of a precoloring $\varphi$ of the cycle $C$ bounding the outer face of a hued patch in general.
However, a useful necessary condition (Lemma~\ref{lemma-topol}) can be expressed in terms of the homotopy of the walk to
which $C$ is mapped by the homomorphism of $C^\varphi$ to the dappled triangular grid $\grid$.
In the language of algebraic topology, we are studying homotopy to the 1-dimensional complex~$\grid$. 
We prefer a self-contained treatment, so let us start with some definitions.

A \emph{closed walk} $U$ in a simple graph is a cyclic sequence $u_1,\ldots,u_{m}$ such that $u_{i}u_{i+1}$ is an edge for $i\in\{1,\ldots,m-1\}$
and $u_mu_1$ is an edge.  Let us remark that the labelling of vertices in the cyclic sequence is
arbitrary, i.e., $u_1$ and $u_m$ do not play any special role and are not specified with the closed walk.
By a \emph{walk}, we mean a sequence
$W=v_1,\ldots,v_k$ of vertices such that $v_iv_{i+1}$ is an edge for $i\in\{1,\ldots,k-1\}$.  We say that $v_{1}$ and $v_{k}$ are the \emph{ends} of the walk.
If $v_1=v_k$, we say that $W$ is an \emph{opening} of the closed walk given by the cyclic sequence $v_1, \ldots, v_{k-1}$.
Note that if the vertices of a closed walk $U=u_1,\ldots,u_m$ are pairwise different, then $U$ has exactly $m$ different openings.
For a closed walk $U$, we write $-U$ for the reversed closed walk.

We say that a walk $W$ is a \emph{combination} of a walk $W_1=v_1,\ldots, v_k$ and a closed walk $W_2$ if
there exists an opening $z_1,\ldots,z_m$ of $W_2$ and
$i\in \{1,\ldots,k\}$ such that $v_i=z_1=z_m$ and $W=v_1,\ldots, v_i, z_2\ldots, z_m, v_{i+1},\ldots, v_k$.
A combination of two closed walks is defined analogously.

Let $W=v_1,v_2,\ldots, v_k$ be a walk.  If $v_i=v_{i+2}$
for some $i\in\{1,\ldots,k-2\}$, we say that the segment $v_iv_{i+1}v_{i+2}$ of the walk is \emph{retractable},
and that $W'=v_1,\ldots, v_iv_{i+3},\ldots, v_k$ is obtained from $W$ by a \emph{one-step retraction}.
The \emph{topological retract} of $W$ is the walk obtained by performing
the one-step retraction operation as long as possible; note that the topological retract of $W$ has the same ends as $W$.  
\begin{observation}\label{obs-unique}
The topological retract of a walk is unique, independent of the choice and order of the one-step retraction operations.
\end{observation}
\begin{proof}
Suppose for a contradiction that there exists a walk $W$ and two sequences $R=r_1,\ldots,r_a$ and $R'=r'_1,\ldots,r'_b$ of one-step retraction operations
ending in different topological retracts $Z$ and $Z'$.  Choose $W$ and these sequences to be as short as possible.
Let $x_1x_2x_3$ be the retractable segment of $W$ contracted by $r_1$.  This segment does not appear in $Z'$, and thus at least one
of its edges must be removed by a one-step retraction operation of $R'$.  Thus, there exists $j\in \{1,\ldots,b\}$ such that,
letting $W'$ be the walk obtained from $W$ by performing the operations $r'_1$, \ldots, $r'_{j-1}$, $W'$ still contains the segment $x_1x_2x_3$ and
\begin{itemize}
\item $r'_j$ contracts the segment $x_1x_2x_3$, or
\item $W'$ contains segment $x_1x_2x_3x_4$ with $x_2=x_4$ and $r'_j$ contracts the retractable segment $x_2x_3x_4$, or
\item $W'$ contains segment $x_0x_1x_2x_3$ with $x_0=x_2$ and $r'_j$ contracts the retractable segment $x_0x_1x_2$.
\end{itemize}
Let $R''=r_1,r'_2,\ldots,r'_{j-1},r'_{j+1},\ldots,r'_b$, and observe that in any of the three cases, performing the sequence $R''$ on $W$ has the same
effect as performing $R'$, i.e., results in $Z'$.  Let $W_1$ be obtained from $W$ by performing $r_1$; then $Z$ and $Z'$ are obtained
from $W_1$ by performing $R_1=r_2,\ldots,r_a$ and $R''_1=r'_2,\ldots,r'_{j-1},r'_{j+1},\ldots,r'_b$.  This contradicts the choice of the sequences $R$ and $R'$ as shortest possible.
\end{proof}

We define a one-step retraction and the topological retract for closed walk analogously;
to clarify, we define a one-step retraction of a closed walk of length two to be an empty walk.
We say that a closed walk is \emph{null-homotopic} if its topological retract is empty.
Note that a null-homotopic closed walk necessarily has even length.
The following claim has the same proof as Observation~\ref{obs-unique}.
\begin{observation}\label{obs-unique-closed}
The topological retract of a closed walk is unique, independent of the choice and order of the one-step retraction operations.
\end{observation}
We need several further simple observations on topological retracts.

\begin{observation}\label{obs-two}
If $W$ is a non-empty null-homotopic closed walk, then there exist at least two segments of $W$ are retractable.
\end{observation}
\begin{proof}
By induction on the length of $W$.  The claim is obvious if $|W|=2$.  Otherwise, consider a retractable segment $x_1x_2x_3$ of $S_1$
of $W$, and let $W'$ be the one-step retraction of $W$ on this segment.  By the induction hypothesis, $W'$ has
at least two retractable segments; hence, there is a retractable segment $S_2$ of $W'$ whose middle vertex is not $x_1$.
Then $S_1$ and $S_2$ are retractable segments of $W$.
\end{proof}

\begin{observation}\label{obs-compnull}
A combination $Z$ of a closed walk $W$ with a null-homotopic closed walk $W'$ is null-homotopic
if and only if $W$ is null-homotopic.
\end{observation}
\begin{proof}
The claim is trivial if $W'$ is empty.  Otherwise, by Observation~\ref{obs-two}, there exists a retractable segment $S$ of $W'$
that also appears in the opening of $W'$ used to construct $Z$, and thus also in $S$.  Let $Z_1$ and $W'_1$ be the one-step
retractions of $Z$ and $W'$ on $S$, respectively; then $Z_1$ is a combination of $W$ and $W'_1$.  The claim now follows
by induction, using the uniqueness of the topological retracts.
\end{proof}

\begin{observation}\label{obs-decwalk}
Let $W=u_1,\ldots,u_m$ be a closed walk and let $W_1$ and $W_2$ be walks such that 
$W_1=u_1,u_2,\ldots,u_k$ and $W_2=u_1,u_m,u_{m-1},\ldots, u_k$ for some $k\in\{1,\ldots, m\}$.
Then $W$ is null-homotopic if and only if $W_1$ and $W_2$ have the same topological retract.
\end{observation}
\begin{proof}
If $W_1$ and $W_2$ have the same topological retracts $W'_1$ and $W'_2$, then the concatenation $Z$ of $W'_1$ and $W'_2$
is obtained from $W$ by a sequence of one-step retractions, and since $W'_1=W'_2$, we can retract $Z$ from either end to an empty closed
walk, showing that $W$ is null-homotopic.

Conversely, suppose that $W$ is null-homotopic.  If there exists $i\in\{1,2\}$ such that $W_i$ contains a retractable segment $S$, then we can perform a one-step retraction on $S$
in $W_i$ and in $W$ and proceed by induction.  Otherwise, neither of the two retractable segments $S_1$ and $S_2$ which exist in $W$ appear in $W_1$ and $W_2$,
and thus $u_1$ and $u_k$ are their middle points, respectively.  This implies that $W_1$ and $W_2$ have the same starting edge, i.e., $u_2=u_m$.
We can contract $S_1$ in $W$ and apply the same observation to walks $W''_1=u_2,\ldots,u_k$ and $W''_2=u_m,u_{m-1},\ldots, u_k$ to conclude $u_3=u_{m-1}$, etc.,
eventually concluding that $W_1=W_2$.
\end{proof}

\begin{observation}\label{obs-plane}
Let $H$ be a connected plane graph with the outer face bounded by a cycle $C$ (in clockwise order).
For each internal face $p$ of $H$, let $F_p$ be the closed walk bounding $p$ (in clockwise order).
There exists a combination of these closed walks (in some order) and $-C$ such that the resulting closed walk is null-homotopic.
\end{observation}
\begin{proof}
The claim is obvious if $H=C$, since in that case $H$ has only one internal face $p$ and $F_p=C$.

Otherwise, there exists an edge $e$ incident with two internal faces $p_1$ and $p_2$ such that
$H-e$ is connected.  Let $p$ be the face of $H-e$ containing $p_1$ and $p_2$.  Note that $F_p$
is obtained from a combination $W$ of $F_{p_1}$ and $F_{p_2}$ by a one-step retraction eliminating the two
appearances of the edge $e$.  By induction, there exists a combination of $-C$
and closed walks $F_{p'}$ for internal faces $p'$ of $H$ such that the resulting closed walk $Z_1$
is null-homotopic.  We can replace $F_p$ in this combination by $W$ while keeping the two appearances of $e$
next to each other in the resulting closed walk $Z_2$.  Since $Z_1$ is obtained from $Z_2$ by the one-step
retraction eliminating the two appearances of $e$, $Z_2$ is also null-homotopic.  Since $W$ is a combination of $F_{p_1}$ and $F_{p_2}$,
the claim of the lemma follows.
\end{proof}

For a closed walk $Z$ in a graph $H$ and a homomorphism $f$ from $H$ to another graph $F$,
observe that $f$ maps $Z$ to a closed walk $f(Z)$ in $F$.
\begin{observation}\label{obs-null}
Let $Z$ be a closed walk in a graph $H$ and let $f$ be a homomorphism from $H$ to a graph $F$.
If $Z$ is null-homotopic, then $f(Z)$ is also null-homotopic.
\end{observation}
\begin{proof}
If a segment $S=x_1x_2x_3$ of $Z$ is retractable, i.e., $x_1=x_3$,
then $f(x_1)=f(x_3)$, and thus the segment $f(S)=f(x_1)f(x_2)f(x_3)$ of $f(Z)$ is also retractable.
Moreover, if $Z'$ is the one-step retraction of $Z$ on $S$, then
$f(Z')$ is the one-step retraction of $f(Z)$ on $f(S)$.  The claim then easily follows by induction.
\end{proof}

Let $G$ be a hued graph.  For a closed walk $Z$ in $G$, let $\overline{Z}$ be the subgraph
of $G$ formed by the vertices and edges traversed $Z$.  For a viable 4-coloring $\varphi$ of $\overline{Z}$,
let $\WW(Z,\varphi)$ be the set of the closed walks $f(Z)$ over all homomorphisms $f$ from $\overline{Z}^\varphi$
to the dappled triangular grid $\grid$.  Note that since $f$ is unique
up to the choice of the image of a vertex, all these walks are translations of one another.
We say that $\varphi$ is \emph{null-homotopic on $Z$} if the walks in $\WW(Z,\varphi)$ are null-homotopic.
Let $\WW(Z)$ be the union of
$\WW(Z,\varphi)$ over all viable 4-colorings of $\overline{Z}$. Note that $\WW(Z)$ consists of all closed
walks in $\grid$ with length equal to the length of $Z$ and such that the hues of their vertices match
the hues of vertices of $Z$ in order.

If $G$ is a connected plane graph,
then for each face $p$ of $G$, let $\WW(p)=\WW(Z)$ for the closed walk $Z$ bounding $p$.
Given sets $\WW_1, \ldots, \WW_n$ of closed walks in $\grid$, let $\bigoplus(\WW_1,\ldots,\WW_n)$
denote the set of closed walks that can be obtained as a combination of a walk from $\WW_1$,
a walk from $\WW_2$, \ldots, and a walk from $\WW_n$, in any order.  For a set $\WW$ of closed
walks, let $-\WW=\{-W:W\in\WW\}$.

\begin{lemma}\label{lemma-topol}
Let $G$ be a hued patch with the outer face bounded by a cycle $C$ and let $\varphi$ be a viable 4-coloring of $C$.
If $\varphi$ extends to a 4-coloring of $G$, then for every connected subgraph $H$ of $G$ containing $C$,
$$\WW=\bigoplus(-\WW(C,\varphi), \WW(p)\text{ for all internal faces $p$ of $H$})$$
contains a null-homotopic closed walk.
\end{lemma}
\begin{proof}
Suppose $\varphi$ extends to a 4-coloring $\beta$ of $G$, and let
$f$ be the corresponding homomorphism from $G^\beta$ to $\grid$.  Let $\beta'$ and $f'$
be the restrictions of $\beta$ and $f$ to $H$; then $f'$ 
is a homomorphism from $H^{\beta'}$ to $\grid$.  By Observation~\ref{obs-plane},
there is a null-homotopic composition of $-C$ and the boundary walks of internal faces of $H$,
and by Lemma~\ref{obs-null}, the corresponding composition of the images of these closed walks
under $f'$ (which belongs to $\WW$) is null-homotopic.
\end{proof}

Let us now give a simple application.  A plane graph $H$ is a \emph{near-quadrangulation}
if it is $2$-connected and all internal faces have length four.

\begin{corollary}\label{cor-quanull}
Let $H$ be a near-quadrangulation and let $G$ be the patch extension of $H$.
If a viable 4-coloring $\varphi$ of the boundary of the outer face of $G$ extends to a 4-coloring of $G$,
then $\varphi$ is null-homotopic on the cycle $C$ bounding the outer face of $G$.
\end{corollary}
\begin{proof}
Observe that if $K$ is a hued 4-cycle using only two values of hue, then $\WW(K)$ only contains null-homotopic walks. Recall that by definition, this is the case for $H$.
Hence, for each internal face $p$ of $H$, $\WW(p)$ only contains null-homotopic walks.
By Lemma~\ref{lemma-topol} and Observation~\ref{obs-compnull}, this implies that $\WW(C,\varphi)$ contains
a null-homotopic closed walk, and thus $\varphi$ is null-homotopic on $C$.
\end{proof}

The condition from Lemma~\ref{lemma-topol} is necessary but not sufficient.  However, for the special case of 
the patch extension of a near-quadrangulation $H$, it is easy to strengthen it to a sufficient condition.
Let $C$ be the cycle bounding the outer face of $H$.  We say a path $P$ in $H$
is a \emph{generalized chord} of $C$ with \emph{base $B$} if $P$ intersects $C$ exactly
in the ends of $P$, and $B$ is a subpath of $C$ with the same ends (note that there are two
possible choices for a base of any given generalized chord).
Let $\varphi$ be a viable $4$-coloring of $C$ and let $f$ be the corresponding homomorphism to
the dappled triangular grid $\grid$.  A \emph{$\varphi$-shortcut} is a generalized chord $P$ of $C$ with base $B$
such that the topological retract of $f(B)$ has more edges than $P$.  Note that by Observation~\ref{obs-decwalk},
if $P$ is a $\varphi$-shortcut with base $B$ and $\varphi$ is null-homotopic on $C$, then
$P$ is also a $\varphi$-shortcut for the other possible base of $P$.

\begin{lemma}\label{lemma-noshort}
Let $H$ be a near-quadrangulation with the outer face bounded by a cycle $C$ and let $G$ be the patch extension of $H$.
A viable 4-coloring $\varphi$ of $C$ extends to a 4-coloring of $G$ if and only if $\varphi$ is null-homotopic on $C$ and
$H$ has no $\varphi$-shortcuts.
\end{lemma}
\begin{proof}
Suppose first that $\varphi$ extends to a 4-coloring $\beta$ of $G$,
and let $f$ be the corresponding homomorphism from $G^\beta$ to the dappled triangular grid $\grid$.
By Corollary~\ref{cor-quanull}, $\varphi$ is null-homotopic on $C$, and thus it suffices to show that $H$ has no $\varphi$-shortcuts.
Let $P$ be a generalized chord of $C$ with base $B$, let $H'$ be the subgraph of $H$ drawn in the closed disk bounded
by the cycle $C'=P\cup B$, and let $G'$ be the patch extension of $H'$, considered as a subgraph of $G$.
Let $\varphi'$ be the restriction of $\beta$ to $C'$.  Since $\varphi'$ extends to a 4-coloring of $G'$,
Corollary~\ref{cor-quanull} implies $\varphi'$ is null-homotopic on $C'$.  By Observation~\ref{obs-decwalk}, $f(P)$
has the same topological retract as $f(B)$, and thus $P$ has at least as many edges as the topological retract of $f(B)$.

We prove the converse by induction on $|E(H)\setminus E(C)|$.
Let $f$ be the homomorphism from $C^\varphi$ to the dappled triangular grid $\grid$.
If $H=C$, then $|C|=4$ and since $\varphi$ is null-homotopic on $C$, 
$f$ maps two vertices of $C$ to the same vertex of $\grid$.  Consequently, $\varphi$ uses at most three
colors on $C$, and since $|V(G)\setminus V(C)|=1$, $\varphi$ extends to a 4-coloring of $G$.
Hence, we can assume that $H\neq C$.

Suppose next that there exists a generalized chord $P$ of $C$ with base $B$ such that
the topological retract $R$ of $f(B)$ has exactly $|E(P)|$ edges.  Then $f$ has a unique extension to $P$
such that $f(B\cup P)$ is null-homotopic---if $P=v_0,v_1,\ldots, v_k$ and
$R=u_0,\ldots, u_k$, where $f(v_0)=u_0$ and $f(v_k)=u_k$, we set $f(v_i)=u_i$ for $i=1,\ldots, k-1$.
Extend $\varphi$ to $P$ by giving each vertex $v\in V(P)$ the color of $f(v)$.
Let $B'$ be the subpath of $C$ from $v_0$ to $v_k$ edge-disjoint from $B$.  By Observation~\ref{obs-decwalk},
$f(B)$ and $f(B')$ have the same topological retract, and consequently $f(B'\cup P)$ is null-homotopic as well.
Let $C_1=B\cup P$ and $C_2=B'\cup P$.  For $i\in\{1,2\}$, let $H_i$ and $G_i$ be the subgraphs of $H$ and $G$ drawn in the closed disk bounded by $C_i$;
note that $G_i$ is the patch extension of $H_i$ and that $\varphi$ is null-homotopic on $C_i$.

We claim that $H_i$ has no $\varphi$-shortcuts: By symmetry, we can assume that $i=1$.
Suppose for a contradiction that $Q$ is a $\varphi$-shortcut of $H_i$.
If both endpoints of $Q$ were contained in $B$, then $Q$ would also be a $\varphi$-shortcut of $H$.
If both endpoints of $Q$ were contained in $P$, then consider the subpath $P'$ of $P$ between the
endpoints of $Q$.  Since $f(P)$ is equal to its topological retract, $f(P')$ is also equal to its retract, and
since $Q$ is a $\varphi$-shortcut, we have $|E(Q)|<|E(P')|$.  But then the path $P''$ obtained from $P$ by
replacing $P'$ by $Q$ is shorter than $P$, and since $R$ has exactly $|E(P)|$ edges,
$P''$ would be a $\varphi$-shortcut of $H$.  Hence, $Q$ joins a vertex $v_a$ of $P$ (with $1\le i\le k-1$)
with a vertex $w$ of $V(B)\setminus V(P)$. 
Let $B_1$ and $B_2$ be the subpaths of $B$ from $v_0$ and $v_k$ to $w$,
and for $j\in \{1,2\}$, let $R_j$ be the topological retract of $f(B_j)$.  Since $R$ is the topological retract of the concatenation of $R_1$ and $R_2$
and both $R_1$ and $R_2$ are already topological retracts, $R$ is a concatenation of a prefix $R'_1$ of $R_1$ with a suffix $R'_2$ of $R_2$.
Since $|E(R'_1)|+|E(R'_2)|=k=a+(k-a)$, we have $|E(R'_1)|\ge a$ or $|E(R'_2)|\ge k-a$. 
By symmetry we can assume that $|E(R'_1)|\ge a$.
Let $P_0=v_0\ldots v_a$ and let $D$ be the base of $Q$ consisting of $P_0$ and $B_1$.  
Since the prefix $R'_1$ of $R_1$ of length at least $a=|E(P_0)|$ is also a prefix of $R$,
and $f(P)=R$, the topological retract of $f(D)$ is obtained from $R_1$ by deleting the first $a$ vertices,
and thus it has length $|E(R_1)|-a$.  Since $Q$ is a $\varphi$-shortcut, it follows that
$|E(Q)|<|E(R_1)|-a$ and $|E(Q\cup P_0)|<|E(R_1)|$.  However, this implies that $Q\cup P_0$ is a $\varphi$-shortcut in $H$,
which is a contradiction.

Therefore, neither $H_1$ nor $H_2$ has a $\varphi$-shortcut, and thus by the induction hypothesis,
$\varphi$ extends to a $4$-coloring of $G_1$ and $G_2$.  This gives an extension of $\varphi$ to a $4$-coloring of $G$.

Finally, suppose that every generalized chord $P$ of $C$ with base $B$ has more edges than the topological retract $R$ of $f(B)$.
Since $H$ is bipartite, the lengths of $P$ and $B$ have the same parity.   Clearly the length of a walk and its topological retract
also has the same parity, and thus $|E(P)|\equiv |E(R)|\pmod 2$.  Consequently, $|E(P)|\ge |E(R)|+2$.
In particular, this implies that $C$ is an induced cycle.

Since $H\neq C$, there exists a vertex $v\in V(H)\setminus V(C)$ with a neighbor $u$ in $C$.
Let $H'$ be the graph obtained from $H$ by cutting from the outer face along the edge $uv$; i.e., $u$ is split into two vertices $u_1$
and $u_2$, where $u_1$ is adjacent to $v$ and the vertices adjacent to $u$ via edges leaving $u$ to the left from $uv$
(as seen from the outer face), and $u_2$ is adjacent to $v$ and the vertices adjacent to $u$ via edges leaving $u$ to the right from $uv$.
Let $\varphi'$ be the 4-coloring of the cycle $C'$ bounding the outer face of $H'$ matching $\varphi$ on $V(C)\setminus \{u\}$,
with $\varphi'(u_1)=\varphi'(u_2)=\varphi(u)$ and with $\varphi'(v)$ chosen to be an arbitrary color different from $\varphi(u)$.
Since $\varphi$ is null-homotopic on $C$, observe that $\varphi'$ is null-homotopic on $C'$.  Let $f'$ be the corresponding homomorphism
from $(C')^{\varphi'}$ to $\grid$.

Suppose now that $H'$ has a $\varphi'$-shortcut $Q$ with base $D$.  If $Q$ does not start in $v$, it is also a $\varphi$-shortcut in $H$, which is a
contradiction.  Otherwise, consider the generalized chord $P=Q+uv$ in $H$ and let $B$ be its base obtained from $D-v$
by replacing $u_1$ or $u_2$ by $u$.  Let $R$ be the topological retract of $f(B)$,
and recall that $|E(P)|\ge |E(R)|+2$.
Note that the topological retract of $f'(D)$ has length at most $|E(R)|+1$, and since $Q$ is a $\varphi'$-shortcut, we have
$|E(P)|-1=|E(Q)|<|E(R)|+1\le |E(P)|-1$.  This is a contradiction.  Therefore, there are no $\varphi'$-shortcuts,
and by the induction hypothesis, $\varphi'$ extends to a 4-coloring of the patch extension $G'$ of $H$.  Since $G$ is obtained
from $G'$ by identifying $u_1$ with $u_2$ and $\varphi'(u_1)=\varphi'(u_2)=\varphi(u)$, we conclude that $\varphi$ extends
to a $4$-coloring~of~$G$.
\end{proof}

It is tempting to ask whether the absence of shortcuts does not imply the extendability to 4-coloring in more general
situations, say just assuming that $H$ is a 2-connected bipartite planar graph.  Indeed, the assumption that $H$ is a near-quadrangulation
is only used in the base case of the induction.  However, this base case also shows why it is needed:  If $H$ is a 6-cycle with vertices
colored $1$, $2$, $1$, $3$, $1$, $4$ in order, the precoloring is null-homotopic on $C$ and yet it cannot be extended to a 4-coloring
of the patch extension of $H$.

\bibliographystyle{plain}
\bibliography{precolbib.bib}
\end{document}